\documentclass[journal, 10pt,]{IEEEtran}
\ifCLASSINFOpdf
\else
\fi
\usepackage{amsmath}
\usepackage{color}
\usepackage{pgfplots} 
\usepackage{pgf}
\usepackage{tikz}
\usetikzlibrary{arrows,automata,patterns,chains}
\usepackage{amssymb}
\usepackage{mathtools}
\usepackage{algorithm}
\usepackage{textcomp}
\usepackage{algpseudocode}
\usepackage{amsfonts}
\usepackage{enumitem}
\usepackage{float}
\usepackage{graphicx}
\usepackage{tkz-graph}
\usepackage{amsthm}

\makeatletter
\renewcommand*\env@matrix[1][*\c@MaxMatrixCols c]{%
  \hskip -\arraycolsep
  \let\@ifnextchar\new@ifnextchar
  \array{#1}}
\makeatother  
\newtheorem{definition}{Definition}
\newtheorem{assumption}{Assumption}
\newtheorem{theorem}{Theorem}
\newtheorem{lemma}{Lemma}

\newtheorem{corollary}{Corollary}
\newfloat{algorithm}{t}{lop}
\newcommand{\subscript}[2]{$#1 _ #2$}
\begin{document}
\title{On the Q-linear convergence of Distributed Generalized ADMM under non-strongly convex function components}
%
%
%
\author{Marie~Maros,~\IEEEmembership{Student Member,~IEEE,}
        and~Joakim~Jald\'{e}n,~\IEEEmembership{Senior Member,~IEEE}
\thanks{The authors are with the department of Information Science and Engineering, KTH Royal Institute of Technology, Stockholm 100 44, Sweden (e-mail: mmaros@kth.se; jalden@kth.se).}
\thanks{This project has received funding from the European Research Council (ERC) under the European Union's Horizon 2020 research and innovation programme (grant agreement No 742648)}}

\markboth{TSIPN}%
{Shell \MakeLowercase{\textit{et al.}}: Bare Demo of IEEEtran.cls for IEEE Journals}
\maketitle
\begin{abstract}
Solving optimization problems in multi-agent networks where each agent only has partial knowledge of the problem has become an increasingly important problem. In this paper we consider the problem of minimizing the sum of $n$ convex functions. We assume that each function is only known by one agent. We show that Generalized Distributed ADMM converges Q-linearly to the solution of the mentioned optimization problem if the over all objective function is strongly convex but the functions known by each agent are allowed to be only convex. Establishing Q-linear convergence allows for tracking statements that can not be made if only R-linear convergence is guaranteed. Further, we establish the equivalence between Generalized Distributed ADMM and P-EXTRA for a sub-set of mixing matrices. This equivalence yields insights in the convergence of P-EXTRA when overshooting to accelerate convergence.
\end{abstract}

\section{Introduction}
Consider a network of $n$ agents in which the agents have as a goal to cooperatively solve the optimization problem
\begin{equation}
\label{eq:original}
\underset{\bar{\mathbf{x}} \in \mathbb{R}^p}{\text{min}} \qquad \bar{f}(\bar{\mathbf{x}}) = \sum_{i=1}^n f_i(\bar{\mathbf{x}}).
\end{equation}
The variable $\bar{\mathbf{x}}$ is common to all agents and each function component $f_i: \mathbb{R}^p \to \mathbb{R}$ is a convex function that is known only to agent $i.$ Decentralized or distributed methods provide procedures according to which the agents cooperatively solve \eqref{eq:original} without explicitly exchanging their individual function components. Distributed methods typically rely on a reformulation of \eqref{eq:original} via the introduction of a local copy $\mathbf{x}_i$ of $\bar{\mathbf{x}}$ for each agent $i.$ The agents iteratively update their own copies using both local information and information from their immediate neighbors. The optimization problem in \eqref{eq:original} arises in many applications such as decentralized machine learning, state estimation in smart grids and wireless communications, see \cite{admm}, \cite{application} and references therein.

The alternating direction method of multipliers (ADMM) has attracted a lot of attention in signal processing thanks to its ability to deal with large scale optimization problems \cite{application,sp_admm}, non-smooth decentralized problems \cite{com_admm}, and decentralized problems \cite{large_admm} in the form of \eqref{eq:original}. In the context of \eqref{eq:original} ADMM yields a distributed algorithm that converges at a $\mathcal{O}(\frac{1}{k})$ rate under very mild assumptions on $f_i$ \cite{admm}. Moreover, it is known to converge regardless of step-size \cite{admm}. 
Under the stronger assumptions that the function components $f_i$ $i=1,\hdots,n$ are strongly convex and have Lipschitz continuous gradients, distributed-ADMM (D-ADMM) is known to converge Q-linearly, also regardless of step-size.
Q-linear convergence allows for tracking statements such as those given in \cite{track}.  Stronger assumptions on the problem's variation in time are required if only R-linear, local linear convergence or weaker convergence statements are provided \cite{maros}.  A proof of Q-linear convergence when only the sum in \eqref{eq:original}, $\bar{f},$ is strongly convex does not exist in the literature. Examples in which $\bar{f}$ is strongly convex but the function components $f_i$ are not can be found in \cite{extra}. A simple example in which this is the case is the distributed least squares problem
\begin{equation}
\label{eq:least_squares_distributed}
\underset{\bar{\mathbf{x}} \in \mathbb{R}^{p}}{\text{min}} \quad \frac{1}{2}\sum_{i=1}^n \|\mathbf{h}_i^T\bar{\mathbf{x}} - y_i \|,
\end{equation} 
where each node $i$ knows its measurement vector $\mathbf{h}_i \in \mathbb{R}^{p}$ and has independently obtained the measurement $y_i \in \mathbb{R}^1.$ If the problem \eqref{eq:least_squares_distributed} were to be solved in a centralized manner we would equivalently rewrite the problem as
\begin{equation}
\label{eq:least_squares}
\underset{\bar{\mathbf{x}} \in \mathbb{R}^n}{\text{min}} \quad \frac{1}{2}\|\mathbf{H}\bar{\mathbf{x}} - \mathbf{y}\|^2,
\end{equation}
where $\mathbf{H} \triangleq [\mathbf{h}_1^T;\hdots \mathbf{h}_n^T]$ and $\mathbf{y} \triangleq [y_1,\hdots,y_n].$ Note that if $\mathbf{H}$ is full ranked the optimization problem \eqref{eq:least_squares} is strongly convex. However, for $p > 1$ each of the functions $\frac{1}{2}\|\mathbf{h}_i^T\bar{\mathbf{x}} - y_i\|^2$ is not strongly convex.

Optimal algorithms to solve \eqref{eq:original} in a distributed manner have been recently proposed in \cite{optimal}. Their optimal convergence rates for the case in which $\bar{f}$ is non-strongly convex are further analyzed in \cite{copycats}. The obtained convergence rates correspond to those of the centralized methods with a penalty that depends on the network structure. However, neither optimal algorithms nor the optimal convergence rate in the case where $\bar{f}$ is strongly convex but where the function components $f_i$ are not strongly convex convex are yet known \cite{optimal}. 

Algorithms  for which the linear convergence of the iterates is established given that only $\bar{f}$ is strongly convex do exist \cite{extra}, \cite{interpret}, \cite{pextra}. These include the exact first-order algorithm (EXTRA) \cite{extra} and the distributed inexact gradient tracking method (DIGing) \cite{interpret} and proximal EXTRA (P-EXTRA) \cite{pextra}. EXTRA and DIGing require step-size selection in order to converge. P-EXTRA does not require a specific step-size to converge, but has been empirically observed to converge slower than ADMM \cite{pextra}.

In this paper we establish the global Q-linear convergence of generalized D-ADMM under the assumption that $\bar{f}$ is strongly convex and the functions $f_i$ are convex and $L_i-$smooth. By generalized distributed ADMM we refer to the decentralized algorithm obtained by applying generalized ADMM to an equivalent formulation of \eqref{eq:original}, where generalized ADMM \cite{variational} adds a quadratic perturbation in both primal iterates and an over-relaxation parameter in the dual iterate.
As pointed out in \cite{optimal} Q-linear convergence rates can be achieved in this case by regularizing the objective function. This said, what we show here is that when using generalized distributed ADMM no additional regularizations are necessary to guarantee Q-linear convergence. 
In order to establish this result, we equivalently re-formulate the problem in \eqref{eq:original} in a form that is suitable for the use of the method of multipliers. We then replace the primal iterate of the method of multipliers by an inexact iterate. We do this by upper bounding a term that does not directly allow for a distributable implementation. We then show that the obtained method is equivalent to generalized D-ADMM and use this interpretation to establish the Q-linear convergence of generalized distributed ADMM.

Interpreting the generalized D-ADMM as an inexact version of the method of multipliers also allows us to relate the generalized D-ADMM to other saddle point methods. In particular, we establish that under an appropriate choice of parameters generalized D-ADMM and P-EXTRA yield the same iterates. Hence, we establish that generalized ADMM and P-EXTRA correspond to inexact versions of the method of multipliers that use different upper bounds on a non-distributable term to make the algorithm distributed. This particular analysis allows us to relate overshooting in P-EXTRA to over-relaxation in ADMM and hence provides extended convergence guarantees for P-EXTRA even when overshooting is present. 

In order to simplify the iterates performed by the nodes several decentralized variants of ADMM have been recently proposed \cite{linearize}, \cite{pgadmm}. We believe the equivalence results between generalized distributed ADMM and P-EXTRA extend to PG-ADMM (Proximal Gradient ADMM) \cite{pgadmm} and linearized ADMM, \cite{linearize} and proximal gradient exact first-order algorithm (PG-EXTRA) \cite{pextra} and EXTRA \cite{extra} respectively. If this assertion is true our work establishes a framework to analyze the convergence of these algorithms.

This paper is structured as follows. First, we formulate the constrained problem required for the application of ADMM \cite{dadmm} in Section \ref{section:problem}. We introduce the required notation and conditions to clearly express ADMM as a decentralized method. In Section \ref{section:algorithm} we formulate another problem equivalent to \eqref{eq:original} and apply the Method of Multipliers to solve it. We then show that by generating the appropriate inexact iterates we obtain the generalized distributed ADMM. Section \ref{section:convergence} is devoted to formally proving the convergence properties of generalized distributed ADMM under the strong convexity of $\bar{f}$ and smoothness of the component functions $f_i.$ We discuss the connections between P-EXTRA and the generalized distributed ADMM in Section \ref{section:p-extra}. Section \ref{section:general} provides conditions under which the results from the previous sections can be extended.
\section{Problem Formulation and ADMM \label{section:problem}}
Consider a connected network $\mathcal{G} \triangleq \{\mathcal{V},\mathcal{A}\}$ with vertices $\mathcal{V}$ and arcs $\mathcal{A},$ where each agent corresponds to a vertex. The network consists of $n = |\mathcal{V}|$ vertices and $m = |\mathcal{A}|$ arcs.
 Each agent $i$ is capable of directly communicating with agent $j$ if they are connected via an arc $a(i,j)$, where $a(i,j) \in \{1,\hdots,m\}$ assigns an arc number or label to the arc originating in $i$ and terminating in $j.$
The communication is assumed bidirectional and therefore $(i,j) \in \mathcal{A}$ if $(j,i) \in \mathcal{A}.$
Two adjacent nodes capable of directly communicating will be referred to as neighbors.
Consequently, the neighborhood of node $i$ is defined as $\mathcal{N}_i \triangleq \{j : (i,j) \in \mathcal{A}\}.$

 Let $\mathbf{A}_{\text{s}} \in \mathbb{R}^{mp \times np}$ be the block arc source matrix containing $mn$ square blocks $(\mathbf{A}_{\text{s}})_{a,i} \in \mathbb{R}^{p \times p}.$
  The block $(\mathbf{A}_{\text{s}})_{a,i}$ is set to $\mathbf{0}_{p}$ if arc $a$ does not originate from node $i$ while it is set to $\mathbf{I}_p$ (the $p \times p$ identity matrix) otherwise.
   Analogously, let $\mathbf{A}_{\text{d}}$ denote the block arc destination matrix.
    The extended oriented incidence matrix is defined as $\mathbf{E}_{\text{o}} = \mathbf{A}_{\text{s}} - \mathbf{A}_{\text{d}}$ and the extended unoriented incidence matrix as $\mathbf{E}_{\text{u}} = \mathbf{A}_{\text{s}} + \mathbf{A}_{\text{d}}.$
     Further, let $\mathbf{D} \triangleq \frac{1}{2}(\mathbf{E}_{\text{o}}^T\mathbf{E}_{\text{o}} + \mathbf{E}_{\text{u}}^T\mathbf{E}_{\text{u}})$ denote the extended degree matrix. $\mathbf{D}$ is a diagonal matrix that can be expressed as $\mathbf{D} = \mathbf{D}_{\mathcal{G}} \otimes \mathbf{I}_p,$ where $\mathbf{D}_{\mathcal{G}} \in \mathbb{R}^{n \times n}$ is the degree matrix of $\mathcal{G}.$
      In other words, $\mathbf{D}_{\mathcal{G}}$ is a diagonal matrix in which each diagonal element $d_i$ corresponds to the number of edges connected to agent $i.$ 
      
      Problem \eqref{eq:original} will be reformulated in order to be able to apply ADMM following the steps in \cite{dadmm}.
       First of all, introduce, for each agent, a local copy $\mathbf{x}_i$ of  $\mathbf{\bar{x}}.$
       Then, create the auxiliary variables $\mathbf{z}_{a(i,j)} \in \mathbb{R}^p$ associated with each arc $(i,j) \in \mathcal{A}.$
        Finally the reformulation of \eqref{eq:original} can be written as as
        \begin{subequations}
        \label{eq:def_distributed}
       \begin{align}
        \underset{\{\mathbf{x}_i\} \in \mathbb{R}^p,\{\mathbf{z}_{a(i,j)}\} \in \mathbb{R}^p}{\text{min}}  \quad \sum_{i=1}^n f_i(\mathbf{x}_i) \\
        \text{s.t.}  \quad \mathbf{x}_i = \mathbf{z}_{a(i,j)},\,\mathbf{x}_j = \mathbf{z}_{a(i,j)},\, \forall (i,j) \in \mathcal{A}, \label{eq:def_distributed:constraints}
       \end{align}
       \end{subequations}
 where the constraints $\mathbf{x}_i = \mathbf{z}_{a(i,j)}$ and $\mathbf{x}_j = \mathbf{z}_{a(i,j)}$ enforce consensus among neighboring agents $i$ and $j.$
 A small yet illustrative example is provided in Fig. \ref{fig:graph_example}, together with the corresponding block source and destination matrices. 
\begin{figure}       
\centering
     \begin{tikzpicture}
     \begin{scope}[every node/.style = {circle,very thick, draw = blue}]
     \node  (A) at (0,0) {$\mathbf{x}_1$};
     \node  (B) at (3,0) {$\mathbf{x}_2$};
     \node  (C) at (6,0) {$\mathbf{x}_3$};
     \end{scope}
     \begin{scope}[
              every node/.style={fill=white,circle},
              every edge/.style={draw=blue,very thick}]
     \path [->] (A) edge [bend left = 30] node {$\mathbf{z}_{a(1,2)}$} (B);
     \path [->] (B) edge [bend left = 30] node {$\mathbf{z}_{a(2,1)}$} (A);
     \path [->] (B) edge [bend left = 30] node {$\mathbf{z}_{a(2,3)}$} (C);
     \path [->] (C) edge [bend left = 30] node {$\mathbf{z}_{a(3,2)}$} (B);
     \end{scope}
     \end{tikzpicture}
     
     \begin{tabular}{c|c|c}
     $a(i,j)$&$\mathbf{A_{\text{s}}}$ & $\mathbf{A_{\text{d}}}$ \\ \hline
     $\begin{matrix}
     a(1,2) = 1 \\
     a(2,1) = 2 \\
     a(2,3) = 3 \\
     a(3,2) = 4
     \end{matrix}$
     & $\begin{pmatrix}
     1 & 0 & 0 \\
     0 & 1 & 0 \\
     0 & 1 & 0 \\
     0 & 0 & 1
     \end{pmatrix} \otimes \mathbf{I}_p$  
     & $ \begin{pmatrix}
     0 & 1 & 0 \\
     1 & 0 & 0 \\
     0 & 0 & 1 \\
     0 & 1 & 0
     \end{pmatrix} \otimes \mathbf{I}_p$
     \end{tabular}
     \caption{Above: An example network with three agents. The auxiliary variables can be interpreted as variables lying on the edges. The equality constraints enforce that all variables at either end or on the edge yield the same value. Below: correspondence to arc labels, block source and block destination matrices for $\mathbf{x}_i \in \mathbb{R}^p.$ \label{fig:graph_example}}
\end{figure}
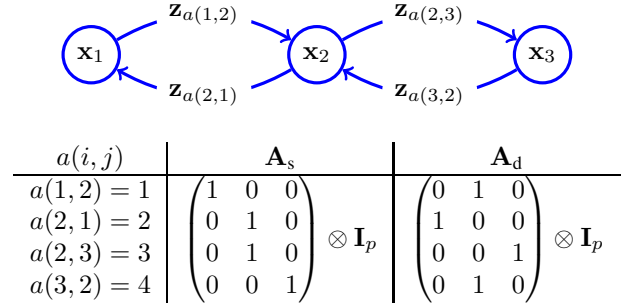
   As long as the network is connected, the constraints in \eqref{eq:def_distributed:constraints} are sufficient to guarantee that \eqref{eq:original} will have the same optimal solution(s) as \eqref{eq:def_distributed} and hence if $\mathbf{\bar{x}}^{\star}$ is an optimal solution to \eqref{eq:original}, $\mathbf{x}_i^{\star} = \mathbf{\bar{x}}^{\star},\forall i \in \mathcal{V}$ is an optimal solution to \eqref{eq:def_distributed}.  

Let $\mathbf{x} \triangleq [\mathbf{x}_1;\hdots;\mathbf{x}_n] \in \mathbb{R}^{np}$ and  $\mathbf{z} \triangleq [\mathbf{z}_{1};,\hdots;\mathbf{z}_m] \in \mathbb{R}^{mp}.$ Further, let $f:\mathbb{R}^{np} \to \mathbb{R}$ be defined as $f(\mathbf{x}) \triangleq \sum_{i=1}^n f_i(\mathbf{x}_i),$ $\mathbf{A} \triangleq [\mathbf{A}_{\text{s}}; \mathbf{A}_{\text{d}}]$ and $\mathbf{B} \triangleq [-\mathbf{I}_{mp};-\mathbf{I}_{mp}].$ Then the optimization problem in \eqref{eq:def_distributed} can then equivalently be written as
\begin{equation}
\underset{\mathbf{x} \in \mathbb{R}^{np},\mathbf{z}\in \mathbb{R}^{mp}}{\text{min}} \,\, f(\mathbf{x}) \qquad \text{s.t.} \, \mathbf{Ax} + \mathbf{Bz} = \mathbf{0}, \label{eq:distributed}
\end{equation}
for which $\mathbf{x}^{\star} = [\mathbf{x}^{\star}_1;\hdots;\mathbf{x}^{\star}_n]$ is optimal.
Note that alternative ways of formulating a distributed version of \eqref{eq:original} that lead to the same form as \eqref{eq:distributed} are possible. These are discussed in detail in \cite{ozdalgar}. The results found in this paper can be extended to other formulations of \eqref{eq:distributed}. This is discussed in Section \ref{section:general}.
Vectors with the same structure as $\mathbf{x}^{\star}$ will be referred to as \emph{consensual}. For clarity we formalize the definition.
\begin{definition}[Consensual]
A vector $\mathbf{v} \in \mathbb{R}^{np}$ partitioned in $p-$length sub-vectors $\mathbf{v} = [\mathbf{v}_1;\hdots;\mathbf{v}_n]$ is consensual if all sub-vectors are identical, i.e. $\mathbf{v}_1 = \hdots = \mathbf{v}_n.$
\end{definition}

Let $\boldsymbol{\alpha}_{a(i,j)}$ denote the dual multipliers associated to the constraint $\mathbf{x}_i = \mathbf{z}_{a(i,j)}$ in \eqref{eq:def_distributed:constraints} and similarly let $\boldsymbol{\beta}_{a(i,j)}$ denote the dual multipliers associated to the constraint $\mathbf{x}_j = \mathbf{z}_{a(i,j)}.$ Group the multipliers $\boldsymbol{\alpha}_{a(i,j)}$ in a vector $\boldsymbol{\alpha}$ and analogously group $\boldsymbol{\beta}_{a(i,j)}$ into $\boldsymbol{\beta}.$ Then, the multiplier $\boldsymbol{\lambda}$ associated to the equality constraint in \eqref{eq:distributed} can be written as $\boldsymbol{\lambda} = [\boldsymbol{\alpha};\boldsymbol{\beta}].$ Let the augmented Lagrangian be defined as
\begin{equation}
\mathcal{L}(\mathbf{x},\mathbf{y},\boldsymbol{\lambda}) \triangleq f(\mathbf{x}) + \boldsymbol{\lambda}^T(\mathbf{Ax} + \mathbf{Bz}) + \frac{\rho}{2} \|\mathbf{Ax} + \mathbf{Bz}\|^2,
\end{equation}
where $\rho > 0.$ Generalized ADMM proceeds by iteratively solving
\begin{subequations}
\label{eq:pure_ADMM}
\begin{align}
& \mathbf{x}^{k+1}:=\text{arg }\underset{\mathbf{x} \in \mathbb{R}^{np}}{\text{min}} \quad \mathcal{L}(\mathbf{x},\mathbf{z}^k,\boldsymbol{\lambda}^k) + \frac{1}{2}\|\mathbf{x}-\mathbf{x}^k\|^2_{\mathbf{P}} \label{eq:pure_ADMM:x}\\
& \mathbf{z}^{k+1}:=\text{arg }\underset{\mathbf{z} \in \mathbb{R}^{mp}}{\text{min}} \quad \mathcal{L}(\mathbf{x}^{k+1}, \mathbf{z},\boldsymbol{\lambda}^k) + \frac{1}{2}\|\mathbf{z}-\mathbf{z}^k\|^2_{\mathbf{Q}}\label{eq:pure_ADMM:z}\\
& \boldsymbol{\lambda}^{k+1} := \boldsymbol{\lambda}^k + \eta\rho (\mathbf{A}\mathbf{x}^{k+1} + \mathbf{B}\mathbf{z}^{k+1}), \label{eq:pure_ADMM:l}
\end{align}
\end{subequations}
where $\mathbf{P} \succeq \mathbf{0},$ $\mathbf{Q} \succeq \mathbf{0}$ and $\eta \in \left(0,\frac{1+\sqrt{5}}{2}\right).$ Note that \eqref{eq:pure_ADMM} becomes the classic ADMM when $\eta =1,$ $\mathbf{P} = \mathbf{0}$ and $\mathbf{Q} = \mathbf{0}.$
The iterates in \eqref{eq:pure_ADMM} for $\mathbf{Q} = \mathbf{0}$ can be manipulated so that their distributed nature is more obvious. In the process of doing so, the effect of the iterate in $\mathbf{z}$ \eqref{eq:pure_ADMM:z} can be incorporated in the iterates in $\mathbf{x}$ \eqref{eq:pure_ADMM:x} and $\boldsymbol{\lambda}$ \eqref{eq:pure_ADMM:l} due to \eqref{eq:pure_ADMM:z} having a closed form solution. Further, under suitable initialization conditions (cf. Assumption \ref{assumption:initialization} in Section \ref{section:algorithm}), instead of updating the multipliers $\boldsymbol{\lambda}^k$ a related variable of smaller dimension, $\boldsymbol{\phi}^k$ is updated. Finally, for the iterate \eqref{eq:pure_ADMM:x} to be distributable across nodes we require that $\mathbf{P} = \boldsymbol{\Pi} \otimes \mathbf{I}_p,$ where $\boldsymbol{\Pi} \in \mathbb{S}^{n}_+$ is a positive semi-definite diagonal matrix with diagonal elements $\pi_i.$ Given the appropriate assumptions on the objective function $f$ Generalized ADMM converges even if $\mathbf{P}$ is indefinite \cite{global_linear}. Then, as showed in \cite{dadmm},
by computing
\begin{subequations}
\label{eq:admm_compare}
\begin{align}
& \mathbf{x}^{k+1} := \text{arg }\underset{\mathbf{x} \in \mathbb{R}^{np}} \quad f(\mathbf{x})+ \label{eq:admmxiterate}\\
& \qquad \left(\mathbf{E}_{\text{o}}^T\boldsymbol{\alpha}^k - \frac{\rho}{2}\mathbf{E}_{\text{u}}^T\mathbf{E}_{\text{u}}\mathbf{x}^k\right)^T\mathbf{x} + \frac{\rho}{2}\|\mathbf{x}\|_{\mathbf{D}}^2 + \frac{1}{2}\|\mathbf{x}-\mathbf{x}^k\|^2_{\mathbf{P}} \nonumber \\
& \boldsymbol{\alpha}^{k+1}:= \boldsymbol{\alpha}^k + \frac{\eta \rho}{2}\mathbf{E}_{\text{o}}\mathbf{x}^{k+1}
\end{align}
for $k \geq 0$ one obtains the same sequence of $\mathbf{x}^k$ iterates as for \eqref{eq:pure_ADMM}. Further, by defining $\boldsymbol{\phi}^k \triangleq \mathbf{E}_{\text{o}}^T\boldsymbol{\alpha}^k$ 
where $\boldsymbol{\phi}^k \triangleq [\boldsymbol{\phi}_1^k;\hdots;\boldsymbol{\phi}_n^k],$
\end{subequations}
the iterates are shown to decouple and can therefore be expressed as
\begin{subequations}
\label{eq:dadmm}
\begin{align}
 \mathbf{x}_i^{k+1}:= \text{arg }\underset{\mathbf{x}_i \in \mathbb{R}^p}{\text{min}} \quad f_i(\mathbf{x}_i) + \label{eq:dadmm_objective}\\ \left(\boldsymbol{\phi}_i^k - \rho \sum_{j \in \mathcal{N}_i}[\mathbf{x}_i^{k} + \mathbf{x}_j^{k}] \right)^T\mathbf{x}_i + \frac{\rho d_i}{2}\|\mathbf{x}_i\|^2 + \frac{\pi_i}{2}\|\mathbf{x}_i - \mathbf{x}_i^k\|^2  \nonumber \\
 \boldsymbol{\phi}_i^k := \boldsymbol{\phi}_{i}^k + \eta \rho \sum_{i \in \mathcal{N}_j}[\mathbf{x}_i^k - \mathbf{x}_j^k] \label{eq:dadmm_dual}.
\end{align}
\end{subequations}

Note that for agent $i$ to be able to update $\boldsymbol{\phi}^{k+1}_i,$ it needs the updates $\{\mathbf{x}_j^{k+1}\}_{j \in \mathcal{N}_i}$ from its neighbors. Hence, after each agent $i$ has updated its own variable $\mathbf{x}_i^{k+1}$ it will broadcast its value within its immediate neighborhood.

We now proceed to introduce the notions we require to establish linear convergence of generalized distributed ADMM (for $\mathbf{Q} = \mathbf{0}$) under the strong convexity of $\bar{f}$ without requiring strong convexity of any of the function components $f_i.$ For simplicity, we will from now on refer to the algorithm in \eqref{eq:dadmm} as generalized D-ADMM.
\section{Approximate Method of Multipliers \label{section:algorithm}}

In this section  we establish that generalized D-ADMM \cite{dadmm} is equivalent to a partially approximated version of the method of multipliers when applied to a suitable equivalent problem (cf. \eqref{eq:equivalent_problem}).
 We establish the equivalence of \eqref{eq:equivalent_problem} and \eqref{eq:distributed} in Lemma \ref{lemma:equivalence}. We then show that by applying the method of multipliers to \eqref{eq:equivalent_problem} and by applying a suitable approximation we recover generalized D-ADMM. Finally, in order to establish the desired convergence result we require that the objective function in problem \eqref{eq:equivalent_problem} is restricted strongly convex (cf. Definition \ref{def:restricted}). We establish that this is the case in Lemma \ref{lemma:restricted}. 

The optimization problem in \eqref{eq:original} can be equivalently formulated as
\begin{subequations}
\label{eq:equivalent_problem}
\begin{align}
\underset{\mathbf{x} \in \mathbb{R}^{np}}{\text{min}} & \quad g(\mathbf{x}) \triangleq f(\mathbf{x}) + \frac{\rho (1-\eta)}{4} \|\mathbf{E}_{\text{o}}\mathbf{x}\|^2, \\
\text{s.t.} & \quad \sqrt{\eta} \mathbf{E}_{\text{o}}\mathbf{x} = \mathbf{0},
\end{align}
\end{subequations}
where $\rho > 0$ and $\eta \in (0,1)$. The problem \eqref{eq:equivalent_problem} has optimality conditions
\begin{subequations}
\begin{align}
\nabla f(\mathbf{x}^{\star}) + \frac{\rho(1-\eta)}{2} \mathbf{E}_{\text{o}}^T\mathbf{E}_{\text{o}}\mathbf{x}^{\star} + \sqrt{\eta} \mathbf{E}_{\text{o}}^T\boldsymbol{\alpha}^{\star} = \mathbf{0} \\
\mathbf{E}_{\text{o}}\mathbf{x}^{\star} = \mathbf{0}.
\end{align}
\end{subequations}

\begin{lemma}[Equivalence of \eqref{eq:distributed} and \eqref{eq:equivalent_problem}] Problems \eqref{eq:distributed} and \eqref{eq:equivalent_problem} are equivalent in the sense that there exists a one to one known mapping between their primal-dual optimal points.
\label{lemma:equivalence}
\end{lemma}
\begin{proof}
Before comparing the optimality conditions of the two problems we need to take a closer look at the structure of $\mathbf{E}_{\text{o}}.$ $\mathbf{E}_{\text{o}}$ can be written as $\mathbf{E}_{\text{o}}^T\mathbf{E}_{\text{o}} = \mathbf{L}_{\mathcal{G}} \otimes \mathbf{I}_{p},$ where $\mathbf{L}_{\mathcal{G}}$ denotes the oriented graph Laplacian matrix. If $\mathcal{G}$ is connected $\mathbf{L}_{\mathcal{G}}$ has rank $n-1$ and $\text{null}\{\mathbf{L}_{\mathcal{G}}\} = \text{span}\{\mathbf{1}_n\}$ \cite{graph}, where $\mathbf{1}_n$ denotes the all ones vector of length $n.$ Based on this, we want to find the set of vectors that span the null space of $\mathbf{E}_{\text{o}}.$ The null space will be spanned by the vectors $\mathbf{x} \in \mathbb{R}^{np}$ such that $\mathbf{E}_{\text{o}}\mathbf{x} = \mathbf{0},$
which is equivalently written as
\begin{equation}
\label{eq:null}
\mathbf{L}_{\mathcal{G}}\mathbf{X} = \mathbf{0},
\end{equation}
where $\mathbf{X} \triangleq [\mathbf{x}_1^T;\hdots;\mathbf{x}_n^T] \in \mathbb{R}^{n\times p}.$ Since $\{\mathbf{L}(\mathcal{G})\} = \text{span}\{\mathbf{1}_n\}$ \eqref{eq:null} will hold true if and only if $\mathbf{x}$ is consensual. Hence, the null space of $\mathbf{E}_{\text{o}}$ is the space of consensual vectors.

We are now ready to compare optimality conditions. A primal dual point $(\mathbf{x}^*,\mathbf{z}^*,\boldsymbol{\lambda}^*)$ is a solution to \eqref{eq:distributed} if it fulfills
\begin{subequations}
\label{eq:distr_optimality}
\begin{align}
\nabla f(\mathbf{x}) + \mathbf{A}^T\boldsymbol{\lambda}^{\star} = \mathbf{0} \\
\mathbf{B}^T\boldsymbol{\lambda}^{\star} = \mathbf{0} \label{eq:half_dual}\\
\mathbf{A}\mathbf{x}^{\star} + \mathbf{B}\mathbf{z}^{\star} = \mathbf{0}. \label{eq:equality}
\end{align}
\end{subequations}
On the other hand, a primal dual optimal point $(\mathbf{x}^{\star},\boldsymbol{\nu}^{\star})$ is a solution to \eqref{eq:equivalent_problem} if it fulfills
\begin{subequations}
\label{eq:equivalent_optimality}
\begin{align}
\nabla f(\mathbf{x}^{\star}) + \frac{\rho(1-\eta)}{2} \mathbf{E}_{\text{o}}^T\mathbf{E}_{\text{o}}\mathbf{x}^{\star} + \sqrt{\eta} \mathbf{E}_{\text{o}}^T\boldsymbol{\nu}^{\star} = \mathbf{0} \label{eq:opt_primal} \\
\mathbf{E}_{\text{o}}\mathbf{x}^{\star} = \mathbf{0}. \label{eq:consensus}
\end{align}
\end{subequations}
Following the notation of Section \ref{section:problem} equation \eqref{eq:half_dual} can be expressed as $\boldsymbol{\alpha}^{*} = - \boldsymbol{\beta}^{*},$ where $\boldsymbol{\lambda}^{\star} = [\boldsymbol{\alpha}^{*};\boldsymbol{\beta}^{*}],$ implying that the optimality conditions \eqref{eq:distr_optimality} can be equivalently written as 
\begin{subequations}
\label{eq:optimal_equivalent}
\begin{align}
\nabla f(\mathbf{x}^{*}) + \mathbf{E}_{\text{o}}^T\boldsymbol{\alpha}^{*} = \mathbf{0} \\
\mathbf{E}_{\text{o}}\mathbf{x}^* = \mathbf{0},
\end{align}
\end{subequations}
where we have used that \eqref{eq:equality} can be equivalently expressed by the pair $\mathbf{A}_s \mathbf{x}^{*} = \mathbf{z}^*$ and $\mathbf{A}_d \mathbf{x}^{*} = \mathbf{z}^{*}$ which can be used to eliminate $\mathbf{z}^{*}$ yielding $\mathbf{E}_{\text{o}}\mathbf{x}^{\star} = \mathbf{0}.$
We now focus our attention on the optimality conditions in \eqref{eq:equivalent_optimality}. \eqref{eq:consensus} allows us to claim that \eqref{eq:opt_primal} can be equivalently written as
\begin{align}
\nabla f(\mathbf{x}^{\star}) + \mathbf{E}_{\text{o}}^T(\sqrt{\eta}\boldsymbol{\nu}^{\star}) = \mathbf{0} \\
\mathbf{E}_{\text{o}}\mathbf{x}^{\star} = \mathbf{0},
\end{align}
which is identical to the optimality conditions in \eqref{eq:optimal_equivalent} except for a scaling factor $\sqrt{\eta}$, i.e. $(\mathbf{x}^{*},\boldsymbol{\alpha}^{*}) = (\mathbf{x}^{\star},\sqrt{\eta}\boldsymbol{\nu}^{\star}).$
\end{proof}
We next establish that the iterates \eqref{eq:dadmm} can be obtained by applying the method of multipliers to \eqref{eq:equivalent_optimality} and then suitably approximating $\|\mathbf{E}_{\text{o}}\mathbf{x}\|^2.$ 
Applying the method of multipliers to \eqref{eq:equivalent_problem} with step size $\frac{\rho}{2}$ yields the iterates
\begin{subequations}
\begin{align}
& \mathbf{x}^{k+1} := \text{arg }\underset{\mathbf{x}}{\text{min}} \quad f(\mathbf{x}) + (\sqrt{\eta} \boldsymbol{\nu}^k)^T \mathbf{E}_{\text{o}}\mathbf{x} + \frac{\rho}{4} \|\mathbf{E}_{\text{o}}\mathbf{x}\|^2, \label{eq:mmprimal} \\
& \boldsymbol{\nu}^{k+1} := \boldsymbol{\nu}^k + \sqrt{\eta} \frac{\rho}{2} \mathbf{E}_{\text{o}}\mathbf{x}^{k+1}.
\end{align}
\end{subequations}
 Note that \eqref{eq:mmprimal} can not be solved in a distributed manner due to the term $\|\mathbf{E}_{\text{o}}\mathbf{x}\|^2$ which can be alternatively written as $ \sum_{i=1}^{n} \sum_{j \in \mathcal{N}_i}\|\mathbf{x}_j-\mathbf{x}_i\|^2.$ In order to be able to carry an iterate in $\mathbf{x}^{k+1}$ in a distributed manner we upper bound $\|\mathbf{E}_{\text{o}}\mathbf{x}\|^2$ for all $\mathbf{x}$ as
\begin{equation}
\|\mathbf{E}_{\text{o}}\mathbf{x}\|^2 \leq \|\mathbf{E}_{\text{o}}\mathbf{x}^k\|^2 + 2\left( \mathbf{E}_{\text{o}}^T\mathbf{E}_{\text{o}}\mathbf{x}^k \right)^T (\mathbf{x} - \mathbf{x}^k) + \|\mathbf{x} - \mathbf{x}^k\|^2_{\boldsymbol{\Gamma}}, 
\end{equation}
with $\boldsymbol{\Gamma} \succeq \mathbf{E}_{\text{o}}^T\mathbf{E}_{\text{o}}.$ In order to obtain this upper bound, we have used the fact that for a convex function $f$ with Lipschitz continuous gradients with Lipschitz modulus $L$ it holds true that
\begin{equation}
f(\mathbf{y}) \leq f(\mathbf{x}) + \nabla f(\mathbf{x})^T(\mathbf{y}-\mathbf{x}) + \frac{L}{2}\|\mathbf{x}-\mathbf{y}\|^2.
\end{equation}
Since $\mathbf{D} = \frac{1}{2}(\mathbf{E}_{\text{o}}^T\mathbf{E}_{\text{o}} + \mathbf{E}_u^T\mathbf{E}_u)$ we choose $\boldsymbol{\Gamma} = 2\mathbf{D} + 2\epsilon\mathbf{P},$ for $\epsilon \geq 0.$ Choosing $\boldsymbol{\Gamma}$ as a diagonal matrix allows for the iterates in $\mathbf{x}^{k+1}$ to be distributed. Hence, we finally have the following algorithm
\begin{subequations}
\label{eq:mmconvergence}
\begin{align}
& \mathbf{x}^{k+1} := \text{arg }\underset{\mathbf{x} \in \mathbb{R}^{np}}{\text{min}} \quad  f(\mathbf{x}) + (\sqrt{\eta}\boldsymbol{\nu}^k)^T\mathbf{E}_{\text{o}}\mathbf{x}  + \label{eq:mmxiterate} \\ & \qquad \frac{\rho}{2} (\mathbf{E}_{\text{o}}^T\mathbf{E}_{\text{o}}\mathbf{x}^k)^T \mathbf{x}  + \frac{\rho}{2} \|\mathbf{x} - \mathbf{x}^k\|_{\mathbf{D} + \epsilon \mathbf{P}}^2 \nonumber\\
& \boldsymbol{\nu}^{k+1} := \boldsymbol{\nu}^{k} + \frac{\rho}{2} \sqrt{\eta} \mathbf{E}_{\text{o}}\mathbf{x}^{k+1}.
\end{align}
\end{subequations}

We will now manipulate the iterate in \eqref{eq:mmxiterate} to make it comparable to \eqref{eq:admmxiterate}. In particular the part of $\frac{\rho}{2} \|\mathbf{x} - \mathbf{x}^{k}\|^2_{\mathbf{D}}$ that depends on $\mathbf{x}$ is $\frac{\rho}{2} \|\mathbf{x}\|_{\mathbf{D}}^2 - \rho (\mathbf{x}^k)^T\mathbf{D}\mathbf{x}.$ Further, using that $2\mathbf{D} = \mathbf{E}_{\text{o}}^T\mathbf{E}_{\text{o}} + \mathbf{E}_{\text{u}}^T\mathbf{E}_{\text{u}}$ we have that the iterate \eqref{eq:mmxiterate} can be equivalently written as
\begin{align}
&\mathbf{x}^{k+1} := \text{arg }\underset{\mathbf{x} \in \mathbb{R}^{np}}{\text{min}} \quad f(\mathbf{x}) + (\sqrt{\eta} \mathbf{E}_{\text{o}}^T\boldsymbol{\nu}^k - \frac{\rho}{2} \mathbf{E}_{\text{u}}^T\mathbf{E}_{\text{u}}\mathbf{x}^k)^T\mathbf{x} \\
& + \frac{\rho}{2} \|\mathbf{x}\|_{\mathbf{D}}^2 + \frac{\rho}{2} \epsilon\|\mathbf{x}-\mathbf{x}^k\|_{\mathbf{P}}^2. \nonumber
\end{align}
If $\sqrt{\eta} \mathbf{E}_{\text{o}}^T\boldsymbol{\nu}^0 = \mathbf{E}_{\text{o}}^T\boldsymbol{\alpha}^0,$ and setting $\epsilon = \frac{1}{\rho}$ we can re-write the equations above as
\begin{subequations}
\label{eq:mmcompare}
\begin{align}
& \mathbf{x}^{k+1} : = \text{arg }\underset{\mathbf{x} \in \mathbb{R}^{np}}{\text{min}} \quad f(\mathbf{x}) + \left(\mathbf{E}_{\text{o}}^T\boldsymbol{\alpha}^k - \frac{\rho}{2}\mathbf{E}_{\text{u}}^T\mathbf{E}_{\text{u}}\mathbf{x}^k\right)^T\mathbf{x} \\
& \qquad \frac{\rho}{2}\|\mathbf{x}\|_{\mathbf{D}}^2 + \frac{1}{2}\|\mathbf{x} - \mathbf{x}^k\|_{\mathbf{P}}^2  \nonumber\\
& \boldsymbol{\alpha}^{k+1} := \boldsymbol{\alpha}^k + \frac{\rho\eta}{2}\mathbf{E}_{\text{o}}\mathbf{x}^{k+1},
\end{align}
\end{subequations}  
which is equivalent to the iterates in \eqref{eq:admm_compare} for $\eta \in (0,1).$

We will now show that $g$ (cf. \eqref{eq:equivalent_problem}) is restricted strongly convex with respect to $\mathbf{x}^{\star}.$ However, for the sake of clarity we first define restricted strong convexity.  
\begin{definition}[Restricted Strong Convexity]\label{def:restricted}
A differentiable and convex function $h(\mathbf{x})$ is restricted strongly convex with respect to a point $\mathbf{\tilde{x}}$ if for all $\mathbf{x}$ it holds that
\begin{equation}
(\nabla h(\mathbf{x}) - \nabla h(\mathbf{\tilde{x}}))^T(\mathbf{x} - \mathbf{\tilde{x}}) \geq \mu_h \|\mathbf{x}-\mathbf{\tilde{x}}\|^2
\end{equation}
for some restricted strong convexity constant $\mu_h > 0.$
\end{definition}
\begin{lemma}[Restricted strong convexity, cf. \cite{extra} \label{lemma:restricted}]
If $\bar{f}$ (cf. \eqref{eq:original}) is strongly convex then $g$ (cf. \eqref{eq:equivalent_problem}) is restricted strongly convex with respect to $\mathbf{x}^{\star}$ with restricted strong convexity constant 
\begin{equation}
\label{eq:restricted_constant}
\mu_g \geq \left\{\frac{\mu_{\bar{f}}}{n} - 2L\gamma, \frac{\tilde{\lambda}_{\text{min}}(\mathbf{E}_{\text{o}}^T\mathbf{E}_{\text{o}})\rho(1-\eta)}{2(1+ \frac{1}{\gamma^2})} \right\},
\end{equation} 
for any $\gamma \in (0,\frac{(\mu_{\bar{f}}/n)}{2L}),$ where $\mu_{\bar{f}}$ denotes the strong convexity constant of $\bar{f}$ and $L$ denotes a Lipschitz modulus of the gradient of $f.$
\end{lemma}

\begin{proof}
This proof is nearly identical to that in Appendix 1 in \cite{extra}, with the difference being that the quantities here are defined as vectors instead of matrices. It is sufficient to notice that one can easily go from one formulation to the other realizing that $\mathbf{E}_o^T\mathbf{E}_o = \mathbf{L}_{\mathcal{G}}\otimes \mathbf{I}_{p}.$ 

By selecting $(1-\eta)\rho = \frac{2(\gamma^2+1)(\mu_{\bar{f}}/n - 2L\gamma)}{\gamma^2\tilde{\lambda}_{\text{min}}(\mathbf{E}_{\text{o}}^T\mathbf{E}_{\text{o}})},$ in \eqref{eq:restricted_constant}  $\mu_g$ can be arbitrarily close to $\frac{\mu_{\bar{f}}}{n}$ if $\gamma$ is made sufficiently small and consequently $(1-\eta)\rho$ sufficiently large.
\end{proof}

We now introduce some results regarding \eqref{eq:dadmm} needed to establish the desired linear convergence result. These results are introduced in the form of Lemmas without proof since they are largely identical to the proofs in \cite{dadmm} and \cite{track}.
\begin{lemma}[cf. \cite{track}, \cite{dadmm}] \label{lemma:column_space} Given an optimal primal solution $\mathbf{x}^{\star}$ of \eqref{eq:distributed} there exist multiple optimal multipliers $\boldsymbol{\lambda}^{\star} = [\boldsymbol{\alpha}^{\star};\boldsymbol{\beta}^{\star}]$ where $\boldsymbol{\alpha}^{\star} = - \boldsymbol{\beta}^{\star}$ such that every $(\mathbf{x}^{\star},\boldsymbol{\lambda}^{\star})$ is a primal-dual optimal pair. Among all these optimal multipliers, there exists a unique $\boldsymbol{\lambda}^{\star} = [\boldsymbol{\alpha}^{\star};\boldsymbol{\beta}^{\star}]$ such that $\boldsymbol{\alpha}^{\star} = -\boldsymbol{\beta}^{\star}$ lies in the column space of $\mathbf{E}_{\text{o}}.$
\end{lemma}
Lemma \ref{lemma:column_space} provides uniqueness of the dual multiplier that lies on the column space of $\mathbf{E}_{\text{o}}.$ This is relevant because if we are capable of confining $\boldsymbol{\lambda}^{k}$ to be equal to $[\boldsymbol{\alpha}^k;-\boldsymbol{\alpha}^k]$ and $\boldsymbol{\alpha}^k$ to lie in the column space of $\mathbf{E}_{\text{o}}$ we will be converging to the unique multiplier discussed in Lemma \ref{lemma:column_space}. Assumption \ref{assumption:initialization} and Lemma \ref{lemma:multipliers} take care precisely of this.
\begin{assumption}[Initialization, cf. \cite{track}, \cite{dadmm}]\label{assumption:initialization}
The multiplier $\boldsymbol{\lambda}$ is initialized as $\boldsymbol{\lambda}^0 = [\boldsymbol{\alpha}^0 ; -\boldsymbol{\alpha}^0],$ with $\boldsymbol{\alpha}^0 \in \mathbb{R}^{mp}$ and $\mathbf{x}^0$ be initialized such that $\mathbf{E}_{\text{o}}\mathbf{x}^0 = 2\mathbf{z}^0.$ Further $\boldsymbol{\alpha}^0$ lies in the column space of $\mathbf{E}_{\text{o}}.$
\end{assumption}
\begin{lemma}[cf. \cite{track}, \cite{dadmm}]\label{lemma:multipliers} Under Assumption \ref{assumption:initialization} the iterates in \eqref{eq:pure_ADMM:l} can be replaced by the iterates $\boldsymbol{\alpha}^{k+1} := \boldsymbol{\alpha}^k + \frac{\rho}{2}\mathbf{E}_{\text{o}}\mathbf{x}^{k+1},$ where $\boldsymbol{\alpha}^k$ converges to $\boldsymbol{\alpha}^{\star}$ as defined in Lemma \ref{lemma:column_space}. Further, the iterates $\boldsymbol{\alpha}^k$ lie in the column space of $\boldsymbol{E}_{\text{o}}.$
\end{lemma}

Assumption \ref{assumption:initialization} confines $\boldsymbol{\lambda}^k = [\boldsymbol{\alpha}^k;-\boldsymbol{\alpha}^k].$ Hence, we only need to update the vector $\boldsymbol{\alpha}$ with lower dimensionality than $\boldsymbol{\lambda}.$ Note that the update $\boldsymbol{\alpha}^{k+1} = \boldsymbol{\alpha}^{k} + \frac{\rho}{2}\mathbf{E}_{\text{o}}\mathbf{x}^{k+1}$ updates $\boldsymbol{\alpha}$ with terms that always lie in the column space of $\mathbf{E}_{\text{o}}.$ This implies, that if the sequence $\{\boldsymbol{\alpha}^k\}$ converges it will do so to the multiplier $\boldsymbol{\alpha}^{\star}$ lying in the column space of $\mathbf{E}_{\text{o}}.$
In \eqref{eq:dadmm} the iterate is in $\boldsymbol{\phi}=\mathbf{E}_{\text{o}}^T\boldsymbol{\alpha}.$ This is because the optimality conditions of the iterate in $\mathbf{x}$ end up depending exclusively on $\mathbf{E}_{\text{o}}^T\boldsymbol{\alpha}.$ 

The next Lemma allows us to directly relate $\mathbf{z}^k$ to $\mathbf{x}^k$ allowing us to skip updating the variable $\mathbf{z}$ entirely.
\begin{lemma}[cf. \cite{track}, \cite{dadmm}] \label{lemma:byez} Under Assumption 1 the iterates $\{\mathbf{x}^k,\mathbf{z}^k\}$ in \eqref{eq:pure_ADMM} satisfy $\mathbf{z}^k = \frac{\mathbf{E}_{\text{u}}}{2}\mathbf{x}^k$.
\end{lemma}

In this section we have established that the iterates of generalized distributed ADMM can be interpreted as an upper bound of the iterates of the method of multipliers. We have also introduced the required notions in order to establish its Q-linear convergence. In the following section we establish the Q-linear convergence of generalized distributed ADMM.
\section{Q-linear Convergence \label{section:convergence}}
This section is devoted to establishing linear convergence rate under the strong convexity of $\bar{f}.$ The section's main statement is provided in Theorem \ref{theorem:linear}. Before introducing the formal statement, we will formalize the conditions under which it holds and define some additional notation. 
\begin{assumption}[Strong Convexity]
The objective function $\bar{f}$ in \eqref{eq:original} is strongly convex with strong convexity constant $\mu_{\bar{f}}.$
\end{assumption}
\begin{assumption}[Lipschitz continuity of function components]
\label{assumption:lipschitz}
All functions $f_i$ have Lipschitz continuous gradients. Further, let $L_i$ denote a Lipschitz modulus of the gradient of $f_i,$ and  $L \triangleq \underset{i}{\text{max}} \,L_i.$ Consequently $L$ is used as the Lipschitz modulus of the gradient of $f.$
\end{assumption}
Let $\mathbf{u}^k \triangleq [\boldsymbol{\alpha}^k; \mathbf{x}^k]$ and $\mathbf{u}^{\star} \triangleq [\boldsymbol{\alpha}^{\star};\mathbf{x}^{\star}]$ denote the primal dual iterate at time $k$ generated by the iterates \eqref{eq:dadmm} and the primal dual optimal point of \eqref{eq:def_distributed} specified in Lemma \ref{lemma:column_space} . Further, let  \begin{equation}
\mathbf{H} \triangleq \begin{pmatrix}
(2/\rho \eta) \mathbf{I}_{mp} & \mathbf{0} \\
\mathbf{0} & \mathbf{M},
\end{pmatrix}
\end{equation}
and $\mathbf{M} \triangleq \frac{\rho}{2}(2  \mathbf{D} + \frac{2}{\rho} \mathbf{P} - \mathbf{E}_{\text{o}}^T\mathbf{E}_{\text{o}})$.

We are now ready to introduce Theorem \ref{theorem:linear}, which constitutes the main contribution of this work. Theorem \ref{theorem:linear} formally states the Q-linear convergence of the primal-dual iterates generated in \eqref{eq:dadmm}. It is worth noting that the statement is made in a semi-norm for any positive semi-definite  $\mathbf{P} \not \succ \mathbf{0}.$ Note that this is the case as well for $\mathbf{P}=\mathbf{0}$ corresponding to the standard ADMM. However, a stronger statement can be made for the $\mathbf{z}$-iterates in \eqref{eq:pure_ADMM}. Such statement is provided in Corollary \ref{corollary:linear}. 

The proof leverages the equivalence between the approximated method of multipliers and ADMM. The problem in \eqref{eq:equivalent_problem} is constructed such that when the method of multipliers is applied, the iterates are equivalent to the using generalized ADMM on the problem \eqref{eq:distributed}
if we select the step size to be $\frac{\rho}{2}$ and have a over-relaxation parameter $\eta$ in the dual update.
 This implies, that by analyzing the convergence of \eqref{eq:mmconvergence} with respect to the optimality conditions of \eqref{eq:equivalent_problem} instead of those of \eqref{eq:distributed}, we are using part of the penalty we would add to the Lagrangian when using the method of multipliers to start from a problem that is better conditioned.    
\begin{theorem}[Global Q-linear convergence] \label{theorem:linear}
Under Assumptions \ref{assumption:initialization}-\ref{assumption:lipschitz}, given that $\mathbf{P} \succeq \mathbf{0},$ $\rho > 0,$ and $\eta \in (0,1),$ the sequence of iterates $\{\mathbf{u}^k\}_{k \geq 0}$ generated by generalized ADMM \eqref{eq:dadmm}  fulfill
\begin{equation}
\|\mathbf{u}^{k+1} - \mathbf{u}^{\star}\|_{\mathbf{H}}^2 \leq \frac{1}{1+\delta}\|\mathbf{u}^{k} - \mathbf{u}^{\star}\|_{\mathbf{H}}^2,
\end{equation}
where $\delta$ is any strictly positive contraction parameter fulfilling
\begin{align}
&\delta \leq   
\underset{\tau > 0}{\text{max }}\text{min} \left\{\frac{ \rho \eta \tilde{\lambda}_{\text{min}}(\mathbf{E}_{\text{o}}^T\mathbf{E}_{\text{o}})}{2(1 + 1/\tau)\lambda_{\text{max}}(\mathbf{M})}, \right. \\
& \left.\frac{\rho \eta\mu_g \tilde{\lambda}_{\text{min}}(\mathbf{E}_{\text{o}}^T\mathbf{E}_{\text{o}})}{(1+\tau)L_g^2 + \rho \eta\lambda_{\text{max}}(\mathbf{M})\tilde{\lambda}_{\text{min}}(\mathbf{E}_{\text{o}}^T\mathbf{E}_{\text{o}})} \right\}
\end{align}
 where $\tilde{\lambda}_{\text{min}}(\mathbf{E}_{\text{o}}^T\mathbf{E}_{\text{o}})$ denotes the smallest non-zero eigenvalue of $\mathbf{E}_{\text{o}}^T\mathbf{E}_{\text{o}},$ $\lambda_{\text{max}}(\mathbf{M})$ denotes the largest eigenvalue of $\mathbf{M}$ and $L_g \triangleq L + (1-\eta)\frac{\rho}{2}\lambda_{\text{max}}(\mathbf{E}_{\text{o}}^T\mathbf{E}_{\text{o}})$ denotes a Lipschitz modulus of the function $g.$
\end{theorem}
\begin{proof}
This proof has many similarities with proofs found in \cite{dadmm}, \cite{track} and \cite{extra}. Further, since the upper bound of $\|\mathbf{E}_{\text{o}}\mathbf{x}\|^2$ can be interpreted as a quadratic approximation it is then also intuitively sound that the proof of convergence is similar to that of \cite{linearize}, where the entire objective function is approximated by a quadratic function at each iterate. We start the proof by leveraging the equivalence between the approximated method of multipliers and generalized D-ADMM.
The optimality condition for \eqref{eq:mmxiterate} can be expressed as
\begin{align}
\nabla f(\mathbf{x}^{k+1}) + \sqrt{\eta} \mathbf{E}_{\text{o}}^T\boldsymbol{\nu}^k + \frac{\rho}{2} \mathbf{E}_{\text{o}}^T\mathbf{E}_{\text{o}}\mathbf{x}^k \\ + \rho(\mathbf{D} + \epsilon \mathbf{P})(\mathbf{x}^{k+1}-\mathbf{x}^k) = \mathbf{0}, \nonumber
\end{align}
which can be equivalently written as
\begin{align}
\label{eq:not_convenient}
\nabla f(\mathbf{x}^{k+1}) + \frac{(1-\eta)\rho}{2} \mathbf{E}_{\text{o}}^T\mathbf{E}_{\text{o}}\mathbf{x}^{k+1} + \sqrt{\eta} \mathbf{E}_{\text{o}}^T\boldsymbol{\nu}^{k+1} = \\
+ \frac{\rho}{2}(2\mathbf{D} + 2\epsilon \mathbf{P} -\mathbf{E}_{\text{o}}^T\mathbf{E}_{\text{o}})(\mathbf{x}^k - \mathbf{x}^{k+1}) \nonumber
\end{align}
where we have used that $\boldsymbol{\nu}^{k} = \boldsymbol{\nu}^{k+1} - \frac{\sqrt{\eta} \rho}{2} \mathbf{E}_{\text{o}}\mathbf{x}^{k+1}$ and added and subtracted $\frac{\rho}{2} \mathbf{E}_{\text{o}}^T\mathbf{E}_{\text{o}}\mathbf{x}^{k+1}.$
Recall that one of the optimality conditions of \eqref{eq:equivalent_problem} is $\nabla f(\mathbf{x}^{\star}) + \frac{(1-\eta)\rho}{2} \mathbf{E}_{\text{o}}^T\mathbf{E}_{\text{o}}\mathbf{x}^{\star} + \sqrt{\eta}\mathbf{E}_{\text{o}}^T\boldsymbol{\nu}^{\star} = \mathbf{0}, $ which we subtract from \eqref{eq:not_convenient} yielding
\begin{align}
\label{eq:optimality_use}
\nabla f(\mathbf{x}^{k+1}) - \nabla f(\mathbf{x}^{\star}) + (1-\eta)\frac{\rho}{2} \mathbf{E}_{\text{o}}^T\mathbf{E}_{\text{o}} (\mathbf{x}^{k+1}- \mathbf{x}^{\star}) = \\  \mathbf{E}_{\text{o}}^T(\boldsymbol{\alpha}^{\star} - \boldsymbol{\alpha}^{k + 1}) + \frac{\rho}{2}(2\mathbf{D} + \frac{2}{\rho} \mathbf{P} -  \mathbf{E}_{\text{o}}^T\mathbf{E}_{\text{o}})(\mathbf{x}^{k} - \mathbf{x}^{k+1}), \nonumber
\end{align}
where we have also used that  $\sqrt{\eta}\boldsymbol{\nu}^{k} = \boldsymbol{\alpha}^k$ and $\sqrt{\eta}\boldsymbol{\nu}^{\star} = \boldsymbol{\alpha}^{\star}.$
By taking the inner product with $(\mathbf{x}^{k+1} - \mathbf{x}^{\star}),$ and using the restricted strong convexity of $g$ (cf. \eqref{eq:equivalent_problem} and Lemma \ref{lemma:restricted}) we have 
\begin{subequations}
\begin{align}
2\mu_g \|\mathbf{x}^{k+1} - \mathbf{x}^{\star}\|^2 & \leq 2(\mathbf{x}^{k+1}-\mathbf{x}^{\star})^T\mathbf{M}(\mathbf{x}^k - \mathbf{x}^{k+1}) + \\ \ &2 (\boldsymbol{\alpha}^{\star} - \boldsymbol{\alpha}^{k+1})\mathbf{E}_{\text{o}}(\mathbf{x}^{k+1} - \mathbf{x}^{\star}). \label{eq:dual_primal}
\end{align}
\end{subequations}

Since $\mathbf{x}^{\star}$ is consensual $\mathbf{E}_{\text{o}}\mathbf{x}^{\star} = \mathbf{0}.$ Further, $ \mathbf{E}_{\text{o}}\mathbf{x}^{k+1} = \frac{2}{\rho\eta}(\boldsymbol{\alpha}^{k+1} - \boldsymbol{\alpha}^k).$ Hence, \eqref{eq:dual_primal} can be equivalently written as
\begin{equation}
\frac{4}{\rho \eta}(\boldsymbol{\alpha}^{k+1} - \boldsymbol{\alpha}^{\star})^T(\boldsymbol{\alpha}^k - \boldsymbol{\alpha}^{k+1}).
\end{equation}
By using the equality $2(\mathbf{u}^{k+1}-\mathbf{u}^{\star})^T\mathbf{H}(\mathbf{u}^k - \mathbf{u}^{k+1}) = \|\mathbf{u}^k - \mathbf{u}^{\star}\|_{\mathbf{H}}^2 - \|\mathbf{u}^{k+1} - \mathbf{u}^{\star}\|_{\mathbf{H}}^2 - \|\mathbf{u}^{k+1} - \mathbf{u}^{k}\|_{\mathbf{H}}^2$ we have that
\begin{align}
2 \mu_g \|\mathbf{x}^{k+1} - \mathbf{x}^{\star}\|^2 + \|\mathbf{u}^{k+1} - \mathbf{u}^{\star}\|_{\mathbf{H}}^2 \\
+ \|\mathbf{u}^{k+1} - \mathbf{u}^k\|_{\mathbf{H}} \leq \|\mathbf{u}^k - \mathbf{u}^{\star}\|_{\mathbf{H}}^2. \nonumber
\end{align}

In order to establish Theorem \ref{theorem:linear} we need to show that
\begin{equation}
2 \mu_g \|\mathbf{x}^{k+1} - \mathbf{x}^{\star}\|^2 + \|\mathbf{u}^{k+1} - \mathbf{u}^k\|_{\mathbf{H}}^2 \geq \delta \|\mathbf{u}^{k+1} - \mathbf{u}^{\star}\|_{\mathbf{H}}^2,
\end{equation}
which is equivalent to showing that
\begin{align}
\label{eq:toshow}
2\mu_g \|\mathbf{x}^{k+1} - \mathbf{x}^{\star}\|^2 + \frac{\rho \eta}{2}\|\mathbf{E}_{\text{o}}(\mathbf{x}^{k+1} - \mathbf{x}^{\star})\|^2 + \|\mathbf{x}^{k+1} - \mathbf{x}^k\|_{\mathbf{M}}^2 \\ \qquad \qquad - \delta \|\mathbf{x}^{k+1} - \mathbf{x}^{\star}\|_{\mathbf{M}}^2 \geq \frac{2\delta}{\rho \eta}\|\boldsymbol{\alpha}^{k+1} - \boldsymbol{\alpha}^{\star}\|^2, \nonumber
\end{align}
where we have used the definition of $\mathbf{u}^k$ and used that $\frac{2}{\rho \eta}\|\boldsymbol{\alpha}^{k+1} - \boldsymbol{\alpha}^{k}\|^2 = \frac{\rho \eta}{2}\|\mathbf{E}_{\text{o}}\mathbf{x}^{k+1} - \mathbf{E}_{\text{o}}\mathbf{x}^{\star}\|^2.$ 
We will now find an upper bound to the RHS of \eqref{eq:toshow} and find the conditions under which this upper bound is bounded above by the LHS of \eqref{eq:toshow}.
For this we rewrite equation \eqref{eq:optimality_use} and take the norm squared on both sides as
\begin{align}
\|\mathbf{E}_{\text{o}}^T(\boldsymbol{\alpha}^{k+1} - \boldsymbol{\alpha}^{\star})\|^2 = \|\nabla f(\mathbf{x}^{k+1}) - \nabla f(\mathbf{x}^{\star}) \\
+ \frac{\rho (1-\eta)}{2} \mathbf{E}_{\text{o}}^T\mathbf{E}_{\text{o}}(\mathbf{x}^{k+1} - \mathbf{x}^{\star}) + \mathbf{M}(\mathbf{x}^{k+1} - \mathbf{x}^k)\|^2,
\end{align}
which can be bounded using the Peter-Paul inequality (Young's inequality with exponent 2) and the fact that $g$ has a $L_g-$ Lipschitz continuous gradient with $L_g = L + (1-\eta)\frac{\rho}{2} \lambda_{\text{max}}(\mathbf{E}_{\text{o}}^T\mathbf{E}_{\text{o}}), $ where $\lambda_{\text{max}}(\mathbf{E}_{\text{o}}^T\mathbf{E}_{\text{o}})$ denotes $\mathbf{E}_{\text{o}}^T\mathbf{E}_{\text{o}}$'s largest eigenvalue. Hence, for any $\tau>0$ it holds that
\begin{align}
\label{eq:matrixproblem}
\|\mathbf{E}_{\text{o}}^T(\boldsymbol{\alpha}^{k+1} - \boldsymbol{\alpha}^{\star})\|^2 \leq (1+\tau)L_g^2\|\mathbf{x}^{k+1} - \mathbf{x}^{\star}\|^2 + \\
\left(1 + \frac{1}{\tau} \right)\|\mathbf{M}(\mathbf{x}^k - \mathbf{x}^{k+1})\|^2.
\end{align}

Further since $\boldsymbol{\alpha}^{k+1}$ and $\boldsymbol{\alpha}^{\star}$ lie in the column space of $\mathbf{E}_{\text{o}}$ (cf. Lemmas \ref{lemma:column_space} and \ref{lemma:multipliers})
we have that $\|\mathbf{E}_{\text{o}}^T(\boldsymbol{\alpha}^{k+1} - \boldsymbol{\alpha}^{\star})\|^2 \geq \tilde{\lambda}_{\text{min}}(\mathbf{E}_{\text{o}}^T\mathbf{E}_{\text{o}})\|\boldsymbol{\alpha}^{k+1} - \boldsymbol{\alpha}^{\star}\|^2,$ where $\tilde{\lambda}_{\text{min}}(\mathbf{E}_{\text{o}}^T\mathbf{E}_{\text{o}})$ denotes the smallest non-zero eigenvalue of $\mathbf{E}_{\text{o}}^T\mathbf{E}_{\text{o}}.$ 

Hence, $\|\boldsymbol{\alpha}^{k+1} - \boldsymbol{\alpha}^{\star}\|^2$ can be upper bounded as
\begin{align}
\label{eq:upper_multipliers}
\|\boldsymbol{\alpha}^{k+1} - \boldsymbol{\alpha}^{\star}\|^2 \leq  &\frac{1+\tau}{\tilde{\lambda}_{\text{min}}(\mathbf{E}_{\text{o}}^T\mathbf{E}_{\text{o}})}L_g^2\|\mathbf{x}^{k+1} - \mathbf{x}^{\star}\|^2 +  \\
& \frac{1+1/\tau}{\tilde{\lambda}_{\text{min}}(\mathbf{E}_{\text{o}}^T\mathbf{E}_{\text{o}})}\|\mathbf{M}(\mathbf{x}^{k+1} - \mathbf{x}^k)\|^2
\end{align}
Let $\mathbf{S} \triangleq 2\mu_g \mathbf{I}_{np} + \frac{\rho \eta}{2} \mathbf{E}_{\text{o}}^T\mathbf{E}_{\text{o}}.$  For the RHS of \eqref{eq:upper_multipliers} to be a lower bound of the LHS of \eqref{eq:toshow} we require that
\begin{align}
\label{eq:condition1}
\|\mathbf{x}^{k+1} - \mathbf{x}^{\star}\|^2_{\mathbf{S}} \geq \frac{2\delta(1+\epsilon)}{\rho \eta \tilde{\lambda}_{\text{min}}(\mathbf{E}_{\text{o}}^T\mathbf{E}_{\text{o}})}L_g^2\|\mathbf{x}^{k+1} - \mathbf{x}^{\star}\| + \\
\qquad \qquad 2\delta\|\mathbf{x}^{k+1} - \mathbf{x}^{\star}\|^2_{\mathbf{M}} \nonumber \\
\|\mathbf{x}^{k+1} - \mathbf{x}^k\|^2_{\mathbf{M}} \geq \frac{2\delta(1+1/\epsilon)}{\rho \eta \tilde{\lambda}_{\text{min}}(\mathbf{E}_{\text{o}}^T\mathbf{E}_{\text{o}})}\|\mathbf{M}(\mathbf{x}^{k+1} - \mathbf{x}^k)\|^2. \label{eq:condition2}
\end{align}
Conditions \eqref{eq:condition1} and \eqref{eq:condition2} will hold true if
\begin{align}
2\mu_g \mathbf{I}_{np} + \frac{\rho \eta}{2}\mathbf{E}_{\text{o}}^T\mathbf{E}_{\text{o}} \succeq \frac{2\delta (1 + \tau)}{\rho \eta \tilde{\lambda}_{\text{min}}(\mathbf{E}_{\text{o}}^T\mathbf{E}_{\text{o}})}L_g^2\mathbf{I}_{np} + 2\delta\mathbf{M} \\
\mathbf{I}_{np} \succeq \frac{2\delta (1 + 1/\tau)}{\rho \eta \tilde{\lambda}_{\text{min}}(\mathbf{E}_{\text{o}}^T\mathbf{E}_{\text{o}})}\mathbf{M},
\end{align}
which are always fulfilled given that $\delta > 0$ is chosen sufficiently small. One can further upper-bound $\delta$ as
\begin{align}
\label{eq:cbextra}
&\delta \leq   
\underset{\tau > 0}{\text{max }}\text{min} \left\{\frac{ \rho \eta \tilde{\lambda}_{\text{min}}(\mathbf{E}_{\text{o}}^T\mathbf{E}_{\text{o}})}{2(1 + 1/\tau)\lambda_{\text{max}}(\mathbf{M})}, \right. \\
& \left.\frac{\rho \eta\mu_g \tilde{\lambda}_{\text{min}}(\mathbf{E}_{\text{o}}^T\mathbf{E}_{\text{o}})}{(1+\tau)L_g^2 + \rho \eta\lambda_{\text{max}}(\mathbf{M})\tilde{\lambda}_{\text{min}}(\mathbf{E}_{\text{o}}^T\mathbf{E}_{\text{o}})} \right\}. \nonumber
\end{align}
\end{proof}

Note that when $\mathbf{P} = \mathbf{0}$ \eqref{eq:optimality_use} can be instead written as
\begin{align}
\nabla f(\mathbf{x}^{k+1}) - \nabla f(\mathbf{x}^{\star}) + (1-\eta)\frac{\rho}{2}\mathbf{E}_{\text{o}}^T\mathbf{E}_{\text{o}}(\mathbf{x}^{k+1} - \mathbf{x}^{\star}) =\\ \mathbf{E}_{\text{o}}^T(\boldsymbol{\alpha}^{\star} - \boldsymbol{\alpha}^{k+1}) + \frac{\rho}{2}\mathbf{E}_{\text{u}}^T\mathbf{E}_{\text{u}}(\mathbf{z}^{k}-\mathbf{z}^{k+1}) \nonumber,
\end{align}
where we have used Lemma \ref{lemma:byez}. Then following the same procedure as in \cite{track} but using the optimality conditions of \eqref{eq:equivalent_problem} instead of \eqref{eq:def_distributed}  we obtain the following Corollary.
\begin{corollary}[\label{corollary:linear} Q-linear Convergence, $\mathbf{P} =\mathbf{0}$ ]
Let $\mathbf{v}^k \triangleq [\boldsymbol{\alpha}^k;\mathbf{z}^k]$ and $\mathbf{v}^{\star} \triangleq [\boldsymbol{\alpha}^{\star};\mathbf{z}^{\star}],$ where $\mathbf{z}^k = \frac{1}{2}\mathbf{E}_{\text{u}}^T\mathbf{E}_{\text{u}}\mathbf{x}^k$ for  $k \geq 0$ and $\mathbf{z}^{\star} = \frac{1}{2}\mathbf{E}_{\text{u}}^T\mathbf{E}_{\text{u}}\mathbf{x}^{\star},$ then for $\mathbf{P} = \mathbf{0},$ $\rho > 0 ,\eta \in (0,1)$ generalized D-ADMM generates iterates that fulfill
\begin{equation}
\|\mathbf{v}^k - \mathbf{v}^{\star}\|^2_{\mathbf{G}} \leq \frac{1}{1+\delta_{\text{ADMM}}}\|\mathbf{v}^{k-1} - \mathbf{v}^{\star}\|^2_{\mathbf{G}},
\end{equation}  
where 
\begin{equation}
\mathbf{G} \triangleq \begin{pmatrix}
\frac{1}{\rho \eta}\mathbf{I}_{mp} & \mathbf{0} \\
\mathbf{0} & \rho \mathbf{I}_{np}
\end{pmatrix}
\end{equation} and $\delta_{\text{ADMM}}$ is a strictly positive contraction parameter fulfilling
\begin{align}
\delta_{\text{ADMM}} \leq \underset{\tau > 0}{\text{max }}\text{min} \left\{ \frac{\eta\tilde{\lambda}_{\text{min}}(\mathbf{E}_{\text{o}}^T\mathbf{E}_{\text{o}})}{(1+1/\tau)\lambda_{\text{max}}(\mathbf{E}_{\text{u}}^T\mathbf{E}_{u})}, \right. \\
\left. \frac{2\rho\eta \mu_g \tilde{\lambda}_{\text{min}}(\mathbf{E}_{\text{o}}^T\mathbf{E}_{\text{o}})}{\rho^2 \eta \lambda_{\text{max}}(\mathbf{E}_{\text{u}}^T\mathbf{E}_{\text{u}})\tilde{\lambda}_{\text{min}}(\mathbf{E}_{\text{o}}^T\mathbf{E}_{\text{o}}) + (1+\tau)L_g^2} \right\}.
\end{align}
\end{corollary}

\section{Equivalence to P-EXTRA \label{section:p-extra}}
In both \cite{interpret} and \cite{pridu} it is established that EXTRA is a saddle point method. More specifically, EXTRA can be shown to perform a gradient descent step in the primal variables and a gradient ascent step in the dual variables of an augmented Lagrangian function. Further, since P-EXTRA is a proximal version of EXTRA it can be analogously interpreted as a saddle point method. It is then unsurprising that, if we interpret generalized D-ADMM as an approximation of the method of multipliers, generalized D-ADMM can be proven to be equivalent to P-EXTRA. 
This section is hence devoted to establishing the equivalence between generalized D-ADMM and P-EXTRA. For this we will introduce P-EXTRA and its requirements for convergence. P-EXTRA's iterates can be written as
\begin{equation}
\label{eq:pextra}
\mathbf{x}^{k+1} := (\mathbf{W} \otimes \mathbf{I}_p)\mathbf{x}^{k} - \xi \nabla f(\mathbf{x}^{k+1}) + \sum_{t=0}^{k}(\mathbf{W} - \tilde{\mathbf{W}})\otimes \mathbf{I}_p \mathbf{x}^t,
\end{equation}
where $\xi > 0$ denotes the step-size and $\mathbf{W}$ and $\tilde{\mathbf{W}}$ are mixing matrices.
The requirements on the matrices are included in the following Assumption.
\begin{assumption}[Requirements on $\mathbf{W}$ and $\tilde{\mathbf{W}},$ \cite{extra}].\label{Assumption:pextra} Consider a connected graph $\mathcal{G}.$ The mixing matrices $\mathbf{W} \in \mathbb{R}^{n \times n}$ and $\tilde{\mathbf{W}} \in \mathbb{R}^{n \times n}$ satisfy:
\begin{enumerate}[label=(\subscript{A}{{\arabic*}})]
\item (Decentralized property) If $i \neq j$ and $(i,j) \not \in \mathcal{A},$ then $\tilde{w}_{ij} = \tilde{w}_{ij} = 0.$
\item (Symmetry) $\mathbf{W} = \mathbf{W}^T,$ $\tilde{\mathbf{W}} = \tilde{\mathbf{W}}^T.$
\item (Null space property) $\text{null}\{\mathbf{W} - \tilde{\mathbf{W}}\} = \text{span}\{\mathbf{1}_n\},$ $\text{null}\{\mathbf{I}_n - \tilde{\mathbf{W}}\} \supseteq \text{span}\{\mathbf{1}_n\}.$
\item (Spectral Property) \label{ass:spectral}$\tilde{\mathbf{W}} \succ \mathbf{0}$ and $\frac{\mathbf{I}_n + \mathbf{W}}{2} \succeq \tilde{\mathbf{W}} \succeq \mathbf{W}.$
\end{enumerate}
\end{assumption}
Now that we have introduced the notation for P-EXTRA we formalize this section's claim in Theorem \ref{theorem:pextra}. The proof of the theorem relies on rewriting the optimality condition of the iterate in $\mathbf{x}$ in \eqref{eq:mmcompare}, \eqref{eq:not_convenient}, and using the fact that $\boldsymbol{\alpha}^k = \frac{\rho \eta}{2}\sum_{t=0}^k \mathbf{E}_{\text{o}}\mathbf{x}^t.$ 
\begin{theorem}[Generalized D-ADMM and P-EXTRA] \label{theorem:pextra}Given that
$\mathbf{P} = \boldsymbol{\Pi} \otimes \mathbf{I}_p,$ is selected such that
\begin{equation}
\boldsymbol{\Pi} = \frac{1}{\xi}(\mathbf{I}_n - \xi \rho \mathbf{D}_{\mathcal{G}})
\end{equation}
which can always be done, the iterates in \eqref{eq:admm_compare} can be expressed as 
\begin{equation}
\label{eq:res_t1}
\mathbf{x}^{k+1} = (\mathbf{W} \otimes \mathbf{I}_p) \mathbf{x}^k - \xi \nabla f(\mathbf{x}^{k+1}) + \sum_{t=0}^{k}(\mathbf{W} - \tilde{\mathbf{W}})\otimes \mathbf{I}_p \mathbf{x}^{t},
\end{equation}
with mixing matrices
\begin{subequations}
\label{eq:mixing}
\begin{align}
\mathbf{W} = \mathbf{I}_n -  \frac{\xi \rho}{2} \mathbf{L}_{\mathcal{G}} \\
\tilde{\mathbf{W}} = \mathbf{I}_n - \frac{\xi \rho}{2}(1-\eta) \mathbf{L}_{\mathcal{G}}.
\end{align}
\end{subequations}
implying that Generalized ADMM and P-EXTRA are equivalent for the given choice of matrices $\mathbf{W},\tilde{\mathbf{W}}$ and $\mathbf{P}.$
\end{theorem}

\begin{proof}
Let us start by re-formulating \eqref{eq:not_convenient} by using that $\sqrt{\eta}\boldsymbol{\nu}^k = \boldsymbol{\alpha}^k$ as
\begin{equation}
\label{eq:admm_pextra}
\nabla f(\mathbf{x}^{k+1}) + \mathbf{E}_{\text{o}}^T\boldsymbol{\alpha}^k + \frac{\rho}{2}\mathbf{E}_{\text{o}}^T\mathbf{E}_{\text{o}}\mathbf{x}^k + \frac{\rho}{2}(2\mathbf{D} + \frac{2}{\rho} \mathbf{P})(\mathbf{x}^{k+1} - \mathbf{x}^k) = \mathbf{0},
\end{equation}
where $\epsilon$ has been chosen to be $\frac{1}{\rho}.$ The expression above can be equivalently written as 
\begin{equation}
\mathbf{x}^{k+1} = \mathbf{W}_1 \mathbf{x}^{k} - \frac{2}{\rho}(2\mathbf{D}+\frac{2}{\rho}\mathbf{P})^{-1}(\nabla f(\mathbf{x}^{k+1})+ \mathbf{E}_{\text{o}}^T\boldsymbol{\alpha}^{k}),
\end{equation}
where $\mathbf{W}_1 \triangleq \mathbf{I}_{np} - (2\mathbf{D} + \frac{2}{\rho}\mathbf{P})^{-1}\mathbf{E}_{\text{o}}^T\mathbf{E}_{\text{o}}.$
By setting $\mathbf{P} = \frac{1}{2\xi}(\mathbf{I}_{np} - \xi \rho \mathbf{D}) = \boldsymbol{\Pi} \otimes \mathbf{I}_p$ and using that $\boldsymbol{\alpha}^k = \sum_{t=0}^k (\frac{\rho \eta}{2})\mathbf{E}_{\text{o}}\mathbf{x}^t$ we have that
\begin{equation}
\mathbf{x}^{k+1} = (\mathbf{I} - \frac{\xi \rho}{2}\mathbf{E}_{\text{o}}^T\mathbf{E}_{\text{o}})\mathbf{x}^k - \xi \nabla f(\mathbf{x}^{k+1}) - \sum_{t=0}^k (\frac{\rho \eta \xi}{2})\mathbf{E}_{\text{o}}^T\mathbf{E}_{\text{o}}\mathbf{x}^t.
\end{equation}
By recalling that $\mathbf{E}_{\text{o}}^T\mathbf{E}_{\text{o}} = \mathbf{L}_{\mathcal{G}} \otimes \mathbf{I}_p,$ we can re-write the iterate as in \eqref{eq:res_t1}. 
\end{proof}

The equivalence established by Theorem \ref{theorem:pextra} allows to establish the convergence of P-EXTRA for cases in which convergence is not guaranteed by the analysis in \cite{pextra}. This is done by applying Theorem \ref{theorem:pextra} and using the results in \cite{global_linear}. Further, the equivalence also showcases that overshooting with generalized ADMM, i.e. selecting $\eta > 1$ corresponds to overshooting with P-EXTRA i.e. selecting $\tilde{\mathbf{W}} \succ \frac{\mathbf{I}_n + \mathbf{W}}{2}$ and therefore guaranteeing the convergence of P-EXTRA when overshooting with the set of matrices provided in Theorem \ref{theorem:pextra}. Finally, using Theorem \ref{theorem:pextra} and the equivalence of generalized ADMM with the inexact method of multipliers we can analyze P-EXTRA under a different optic. Doing so we will now establish the conditions over $\rho $ $\xi$ and $\eta$ for which P-EXTRA is guaranteed to converge according to \cite{pextra} and the conditions under which generalized ADMM is guaranteed to converge according to \cite{global_linear}. 

We start establishing the parameters for guaranteed converge in the case of P-EXTRA. The parameters $\xi, \rho$ and $\eta$ must be chosen such that all the conditions in Assumption \ref{Assumption:pextra} hold. Conditions $(A_1)$-$(A_3)$ hold by simple inspection of the mixing matrices \eqref{eq:mixing}. For $\tilde{\mathbf{W}} \succ \mathbf{0}$ to hold it is required that
\begin{equation}
\label{eq:not_known}
\frac{2}{\xi \rho}\mathbf{I}_n \succ (1-\eta)\mathbf{L}_{\mathcal{G}}.
\end{equation}
Further, for $\frac{\mathbf{I}_n + \mathbf{W}}{2} \succeq \tilde{\mathbf{W}}$ we require that
\begin{equation}
\frac{1}{2} \geq \eta.
\end{equation}
Finally, note that as long as $\eta > 0$ $\tilde{\mathbf{W}} \succeq \mathbf{W}$ always holds. Note that in a distributed set-up the eigenvalues of $\mathbf{L}_{\mathcal{G}}$ can not be easily calculated without additional communication overhead. However $\mathbf{L}_{\mathcal{G}} \succeq 2 \mathbf{D}_{\mathcal{G}}.$ Therefore, we conservatively request
\begin{align}
\frac{2}{\xi \rho} > 2 (1-\eta) \lambda_{\text{max}}(\mathbf{D}_{\mathcal{G}}) \\
\eta \in \left(0,\frac{1}{2} \right],
\end{align}
to guarantee that the mixing matrices \eqref{eq:mixing} fulfill Assumption \ref{Assumption:pextra}. In particular the case which allows the product $\xi \rho$ to be largest corresponds to selecting $\eta = \frac{1}{2}$ which implies $\frac{2}{\xi \rho} > \lambda_{\text{max}}(\mathbf{D}_{\mathcal{G}})$ is the most forgiving case. Note that this also corresponds to having $\frac{\mathbf{I}_n + \mathbf{W}}{2} = \tilde{\mathbf{W}}.$ Further, we can relate P-EXTRA to the approximated method of multipliers discussed in Section \ref{section:algorithm} by setting $\boldsymbol{\Gamma} = \frac{2}{\xi \rho} \mathbf{I}_{np}.$ Hence, it holds that $\boldsymbol{\Gamma} \succeq \mathbf{E}_{\text{o}}^T\mathbf{E}_{\text{o}}.$ This implies that by using the iterates \eqref{eq:pextra} we are solving \eqref{eq:equivalent_problem} using an approximated version of the method of multipliers. In particular, we are upper bounding the non-distributable term $\|\mathbf{E}_{\text{o}}(\mathbf{x} - \mathbf{x}^k)\|^2 \leq \frac{2}{\xi \rho}\|\mathbf{x} - \mathbf{x}^k\|^2.$

We now proceed to establish the analogous result for generalized ADMM. Recall that we have established that ADMM will converge Q-linearly as long as $\mathbf{P} \succeq 0.$ For ADMM we have established Q-linear convergence for
\begin{align}
\eta \in \left(0,1\right) \\
\mathbf{P} \succeq \mathbf{0}
\end{align}
with $\boldsymbol{\Gamma} = 2\mathbf{D} + \frac{2}{\rho} \mathbf{P}.$ Note that this implies that by using generalized ADMM we are solving \eqref{eq:equivalent_problem} using an approximated version of the method of multipliers. However, we are this time upper bounding the non-distributable term $\|\mathbf{E}_{\text{o}}(\mathbf{x} - \mathbf{x}^k)\|^2 \leq \|\mathbf{x} - \mathbf{x}^k\|_{2\mathbf{D}+\frac{2}{\rho}\mathbf{P}}^2.$ Note that in terms of upper bounding $\|\mathbf{E}_{\text{o}}(\mathbf{x} - \mathbf{x}^k)\|^2$ in P-EXTRA all nodes are treated equal, while in the case of ADMM with $\mathbf{P} = \mathbf{0}$ the difference of degree of each of the nodes is taken into account. This can already be seen in \eqref{eq:pextra} and \eqref{eq:admm_pextra} in the scaling of the gradient. In particular, in the case of ADMM, nodes that have more neighbors will take smaller steps towards their optimizer than nodes with fewer neighbors. On the other hand, P-EXTRA takes the conservative stance of scaling the iterates according to the degree of the node with most neighbors.  
Further, note for generalized ADMM we have established Q-linear convergence for the case $\eta \in (1/2,1)$ as well. Therefore, by using Theorem \ref{theorem:pextra} we conclude that P-EXTRA converges Q-linearly also if
\begin{align}
\label{eq:omega}
\tilde{\mathbf{W}} = \mathbf{I}_n + \omega \frac{\xi \rho}{2}\mathbf{L}_{\mathcal{G}},\, \omega \in \left(0.5,1 \right)
\end{align}
which clearly violates $\frac{\mathbf{I}_n+\mathbf{W}}{2} \succeq \tilde{\mathbf{W}}.$ Experimental evidence that EXTRA performs better by violating the condition was provided in \cite{extra} and referred to as overshooting. Further, generalized ADMM has been shown to converge if $\eta \in (0,\frac{1+\sqrt{5}}{2})$ \cite{variational}. Empirically, generalized ADMM has been observed to provide faster convergence for $\eta = 1.618$ implying an even wider overshooting range for P-EXTRA. 

We believe that analogous results can be obtained between the generalized versions of linearized ADMM \cite{linearize} and PG-ADMM \cite{pgadmm} and EXTRA \cite{extra} and PG-EXTRA \cite{pextra} respectively providing a unifying framework under which to analyze the convergence of these algorithms. However, this is out of the scope of this paper.
In retrospect, the convergence results established in Section \ref{section:convergence} can be established by using the convergence analysis of ADMM with step-size $\rho$ on the problem
\begin{align}
\underset{\mathbf{x} \in \mathbb{R}^{np}, \mathbf{z} \in \mathbb{R}^{mp}}{\text{min}} \quad f(\mathbf{x}) + \frac{\rho(1-\eta)}{2}\|\mathbf{Ax}  + \mathbf{Bz}\|^2 \\
\text{s.t.} \quad \sqrt{\eta}(\mathbf{Ax} + \mathbf{Bz}) = \mathbf{0}.
\end{align}
However, the equivalence we established in Section \ref{section:algorithm} allowed us to relate both ADMM and P-EXTRA to an approximated version of the method of multiplier which allowed us to interpret generalized ADMM and P-EXTRA under a different optic. Further, theorem \ref{theorem:pextra} assumes that the problem \eqref{eq:original} is reformulated as \eqref{eq:distributed} following the methodology of \cite{dadmm}. However, as studied in \cite{ozdalgar} other reformulations are possible. Hence, in order to make the presented results more general we introduce, in the final technical section, more general conditions under which all the analysis done until now holds.

\section{General Formulation \label{section:general}}
In this section we introduce different reformulations to \eqref{eq:def_distributed} that allow us to generalize all the results established until now. This extends the Q-linear convergence of generalized D-ADMM to more constraint matrices \eqref{eq:distributed} and generalizes the convergence under overshooting of P-EXTRA to a larger variety of mixing matrices.

As an alternative to the formulation in \eqref{eq:def_distributed} one can formulate the following equivalent optimization problem
\begin{align}
\label{eq:distributed_equivalent}
\underset{\mathbf{x} \in \mathbb{R}^{np},\mathbf{z}\in \mathbb{R}^{mp}}{\text{min}} \quad f(\mathbf{x}) \qquad \text{s.t.}\,\, \bar{\mathbf{A}}\mathbf{x} + \mathbf{B}\mathbf{z} = \mathbf{0},
\end{align}
where $\bar{\mathbf{A}} \triangleq \frac{1}{2}[(\sqrt{\mathbf{U}} + \sqrt{\mathbf{V}}) \otimes \mathbf{I}_{p} ; (\sqrt{\mathbf{U}}-\sqrt{\mathbf{V}})\otimes \mathbf{I}_p],$ and $\sqrt{\mathbf{V}} \in \mathbb{R}^{m \times n}$ and $\sqrt{\mathbf{U}} \in \mathbb{R}^{m \times n}$ are matrices that fulfill $(\sqrt{\mathbf{V}})^T\sqrt{\mathbf{V}} = \mathbf{V}$ and $(\sqrt{\mathbf{U}})^T\sqrt{\mathbf{U}} = \mathbf{U}.$ Further the matrices $\mathbf{V}$ and $\mathbf{U}$ must fulfill the following conditions:
\begin{assumption}[Mixing matrices]\
\begin{enumerate}[label=(\subscript{C}{{\arabic*}})]
\item \label{condition1} \emph{(Nullspace property)} $\text{nullspace}\{\mathbf{V}\} = \text{span}\{\mathbf{1}_n\},$ where $\mathbf{1}_n$ denotes the vector of all ones of length $n,$
\item \label{condition2} \emph{(Complementarity)} $\mathbf{V} + \mathbf{U} = 2\bar{\mathbf{D}},$ where $\bar{\mathbf{D}}$ is any positive definite diagonal matrix,
\item \label{condition3} \emph{(Distributable)} $\mathbf{U}$ and $\mathbf{V}$ fulfill that $\mathbf{U}_{a,i} = \mathbf{V}_{a,i} = 0$ if $a$ is not an edge that connects to $i.$ 
\end{enumerate} 
\end{assumption}
It is easy to verify that \eqref{eq:distributed_equivalent} is equivalent to \eqref{eq:def_distributed} by verifying that the constraint in \eqref{eq:distributed_equivalent} enforces that the solution $\mathbf{x}^{\star}$ to \eqref{eq:distributed_equivalent} is consensual. Let $\mathbf{x}^{\star}$ and $\mathbf{z}^{\star}$ denote a primal optimal solution of \eqref{eq:distributed_equivalent}. Then, the equality constraint in \eqref{eq:distributed_equivalent} can be written as
\begin{equation}
\frac{1}{2}(\sqrt{\mathbf{U}} + \sqrt{\mathbf{V}}) \otimes \mathbf{I}_p \mathbf{x}^{\star} = \frac{1}{2}(\sqrt{\mathbf{U}} - \sqrt{\mathbf{V}}) \otimes \mathbf{I}_p \mathbf{x}^{\star},
\end{equation}
implying that the equality constraint in \eqref{eq:distributed_equivalent} is fulfilled if and only if
\begin{equation}
\label{eq:iff}
(\mathbf{\sqrt{V}} \otimes \mathbf{I}_p) \mathbf{x}^{\star} = \mathbf{0}.
\end{equation}
As was shown in Lemma \ref{lemma:equivalence} \eqref{eq:iff} holds true if and only if $\mathbf{x}^{\star}$ is consensual. 

Further, if we use generalized D-ADMM to solve the problem in \eqref{eq:distributed_equivalent} following the same procedure as in \cite{dadmm} and \cite{track} we obtain the iterates
\begin{subequations}
\label{eq:general_dadmm}
\begin{align}
\mathbf{x}^{k+1}:= \text{arg }\underset{\mathbf{x} \in \mathbb{R}^{np}}{\text{min}} \quad f(\mathbf{x}) + \left(\boldsymbol{\phi}^k - \frac{\rho}{2}(\mathbf{U}\otimes \mathbf{I}_p)\mathbf{x}^k \right)^T\mathbf{x} \\
\qquad + \frac{\rho}{2}\|\mathbf{x}\|_{\bar{\mathbf{D}}\otimes \mathbf{I}_p}^2 + \frac{1}{2}\|\mathbf{x} - \mathbf{x}^k\|_{\mathbf{P}}^2 \nonumber \\
\boldsymbol{\phi}^{k+1} : = \boldsymbol{\phi}^{k} + \frac{\eta\rho}{2} (\mathbf{V} \otimes \mathbf{I}_p) \mathbf{x}^{k+1},
\end{align}
\end{subequations}
which depend exclusively on the matrices $\mathbf{U}$ and $\mathbf{V}$ and not their square roots.
In particular, for the formulation in \eqref{eq:def_distributed} conditions \ref{condition1}-\ref{condition3} are fulfilled by assigning to the matrices $\mathbf{E}_{\text{o}}$ and $\mathbf{E}_{\text{u}}$ the roles of $\sqrt{\mathbf{V}}$ and $\sqrt{\mathbf{U}}$ respectively. In particular, condition \ref{condition1} can be verified by writing $\mathbf{E}_{\text{o}}^T\mathbf{E}_{\text{o}} = \mathbf{L}_{\mathcal{G}} \otimes \mathbf{I}_p,$ where $\mathbf{L}_{\mathcal{G}}$ is the oriented graph Laplacian matrix.

In order to establish the same results that were established in Section \ref{section:convergence} one has to define a problem equivalent to \eqref{eq:distributed_equivalent} and then apply the method of multipliers and approximate the iterates by an upper bound. The equivalent problem to use in this case is
\begin{equation}
\underset{\mathbf{x} \in \mathbb{R}^{np}}{\text{min}} \,\, f(\mathbf{x}) + \frac{\rho (1-\eta)}{4}\|\mathbf{x}\|^2_{\mathbf{V} \otimes \mathbf{I}_p} \,\,\,\, \text{s.t.} \, \sqrt{\eta} \sqrt{\mathbf{V}} \mathbf{x} = \mathbf{0}.
\end{equation}

Finally, under this new formulation, generalized D-ADMM yields iterates equivalent to those of P-EXTRA by selecting the mixing matrices
\begin{subequations}
\begin{align}
\mathbf{W} = \mathbf{I}_n - \frac{\xi \rho}{2}\mathbf{V} \\
\tilde{\mathbf{W}}_n = \mathbf{I} - \frac{\xi \rho}{2}(1-\eta)\mathbf{V}.
\end{align}
\end{subequations}
 


\section{Conclusion}
In this paper we have shown that generalized distributed ADMM converges Q-linearly to the optimal solution even if only the objective function $\bar{f}$ (c.f. \eqref{eq:original}) is strongly convex, i.e. the function components $f_i$ may not be strongly convex. Further, we have established that under appropriate choice of parameters generalized distributed ADMM and P-EXTRA are equivalent. Consequently we related overshooting with P-EXTRA to over-relaxation with generalized ADMM. While an increase in performance for EXTRA was previously observed in experiments by overshooting, the convergence of neither P-EXTRA nor EXTRA were established if overshooting was performed. We provide convergence guarantees for overshooting with P-EXTRA and conjecture that a similar result can be derived for EXTRA through the study of linearized generalized ADMM.





\begin{thebibliography}{1}
\bibitem{track}
Q.~Ling, A.~Ribeiro, ``Decentralized Dynamic Optimization Through the Alternating Direction Method of Multipliers," \emph{IEEE Trans. Signal Process.,}  vol. 62, no. 5, pp. 1185-1197, March 2014.
\bibitem{dadmm}
W.~Shi, Q.~Ling, et. al., ``On the Linear Convergence of the ADMM in Decentralized Consensus Optimization," \emph{IEEE Trans. Signal Process.,} vol. 62, no. 7, pp. 1750-1761, April 2014.
\bibitem{graph}
C.~Godsil, G.~Royle, ``Algebraic Graph Theory," New York, NY, USA: Springer-Verlag, 2001.
\bibitem{admm}
S.~Boyd, N.~Parikh, et. al., ``Distributed Optimization and Statistical Learning via the Alternating Direction Method of Multipliers," \emph{Foundation and Trends\textregistered in Machine Learning,} vol. 3, no. 1, pp.1-122, January 2011.
\bibitem{application}
I.~Necoara, V.~Nedelcu, et. al., ``Parallel and distributed optimization methods for estimation and control in networks," \emph{Journal of Process Control,} vol. 21, no. 5, pp. 756-766, June 2011. 
\bibitem{extra}
W.~Shi, Q.~Ling, et. al., ``EXTRA: An Exact First-Order Algorithm for Decentralized Consensus Optimization," \emph{SIAM J. Optim.,} vol. 25, no. 5, pp. 944-966, May 2015.
\bibitem{interpret}
A.~Nedi\'{c}, A.~Olshevsky, et. al., ``Achieving Geometric Convergence for Distributed Optimization Over Time-Varying Graphs," \texttt{arxiv:1607.03218v3}, March 2017.
\bibitem{pridu}
A.~Mokhtari, A.~Ribeiro, ``DSA: Decentralized Double Stochastic Averaging Gradient Algorithm," \emph{Journal of Machine Learning Research,} vol. 17, no. 1, pp. 2165-2199, January 2016.
\bibitem{maros}
M.~Maros, J.~Jald\'{e}n, ``ADMM for Distributed Dynamic Beam-forming," \emph{IEEE Trans. Signal and Info. Process. over Networks,} To appear.
\bibitem{nesterov}
I.~Necoara, Y.~Nesterov, et. al., ``Linear convergence of first order methods for non-strongly convex optimization," \texttt{arxiv:1504.06298v4}, August 2016.
\bibitem{linearize}
Q.~Ling, W.~Shi, et. al., ``DLM: Decentralized Linearized Alternating Direction Method of Multipliers," \emph{IEEE Trans. Signal Process.,} vol. 63, no. 15, pp. 4051-4064, May 2015.
\bibitem{ozdalgar}
A.~Makhdoumi, A.~Ozdaglar, ``Convergence Rate of Distributed ADMM Over Networks," \emph{IEEE Trans. on Autom. Control,} vol. 62, no. 10, pp. 5082-5095, March 2017.
\bibitem{pextra}
W.~Shi, Q.~Ling, et. al. ``A Proximal Gradient Algorithm for Decentralized Composite Optimization," \emph{IEEE Trans. Signal Process.,} vol. 63, no. 22, pp. 6013-6023, November 2015.
\bibitem{optimal}
K.~Scaman, F.~Back, et. al. ``Optimal algorithms for smooth and strongly convex distributed optimization in networks," \texttt{arxiv:1702.08704v2,} April 2017.
\bibitem{global_linear}
W.~Deng, W.~Yin, ``On the Global and Linear Convergence of the Generalized Alternating Direction Method of Multipliers," \emph{J. Sci. Comput.,} vol. 66, no. 3, pp. 889-916, May 2015. 
\bibitem{variational}
R.~Glowinski, ``Numerical Methods for Nonlinear Variational Problems," Springer Series in Computational Physics. Springer, Berlin, 1984.
\bibitem{sp_admm}
B.~Wahlberg, S.~Boyd, et. al. ``An ADMM Algorithm for a Class of Total Variation Regularized Estimation Problems," 16th IFAC Symposium on System Identification, vol. 45, no. 16, pp. 83-88, July 2012.
\bibitem{com_admm}
C.~Shen, T.~Chang, et. al. ``Distributed Robust Multicell Coordinated Beamforming with Imperfect CSI: An ADMM Approach," \emph{IEEE Trans. Signal Proc.} vol. 60, no. 6, pp. 2988-3003, February 2012.
\bibitem{large_admm}
T.~Chang, M.~Hong, et. al. ``Asynchronous Distributed ADMM for Large-Scale Optimization-Part I: Algorithm and Convergence Analysis, " \emph{IEEE Trans. Signal Proc.,} vol. 64, no. 12, pp. 3118-3130, March 2016. 
\bibitem{copycats}
C.~Uribe, S.~Lee, et. al. ``Optimal Algorithms for Distributed Optimization," \texttt{arxiv:1712.00232,} December 2017.
\bibitem{pgadmm}
N.S.~Aybat, Z.~Wang, et. al. ``Distributed Linearized Alternating Direction Method of Multipliers for Composite Convex Consensus Optimization," \emph{IEEE Trans. Autom. Control,} vol. 63, no. 1, pp. 5- 20, January 2018. \emph{Trans.}
\end{thebibliography}
\end{document}